\documentclass[11pt,oneside]{amsart}
\usepackage{amssymb}
\usepackage{fullpage}
\usepackage{xifthen, color}	

\newtheorem{theorem}{Theorem}
\newtheorem{corollary}[theorem]{Corollary}
\theoremstyle{definition}
\newtheorem{definition}[theorem]{Definition}
\theoremstyle{remark}
\newtheorem{example}[theorem]{Example}

\newcommand{\real}{{\mathbb R}}
\newcommand{\zed}{{\mathbb Z}}
\newcommand{\expect}{{\mathbb E}}
\newcommand{\prob}{{\mathbb P}}
\newcommand{\One}{{\mathbf 1}}

\newcommand{\FLAG}[2]{\left[ \genfrac{}{}{0pt}{}{#1}{#2} \right]}   
\newcommand{\gmf}[1]{\mathcal{M}^{\gamma}_{#1}}    

\newcommand{\dmufloor}[1]{\,\lfloor d\mu_{#1}\rfloor}	
\newcommand{\dmuceil}[1]{\,\lceil d\mu_{#1}\rceil}	

\begin{document}

\title{Hadwiger Integration of Random Fields}
\author{Matthew L. Wright}
\date{\today}
\address{Institute for Mathematics and its Applications, University of Minnesota, Minneapolis, Minnesota, USA}
\email{mlwright@ima.umn.edu}

\begin{abstract}
	Hadwiger integrals employ the intrinsic volumes as measures for integration of real-valued functions.
	We provide a formula for the expected values of Hadwiger integrals of Gaussian-related random fields.
	The expected Hadwiger integrals of random fields are both theoretically interesting and potentially useful in applications such as sensor networks, image processing, and cell dynamics.
	Furthermore, combining the expected integrals with a functional version of Hadwiger's theorem, we obtain expected values of more general valuations on Gaussian-related random fields.
\end{abstract}


\maketitle

\section{Introduction}

The intrinsic volumes are valuations on sets that provide various notions of the \emph{size} of a set, thus generalizing both Lebesgue volume and Euler characteristic.
Because the intrinsic volumes are additive, they can be used as ``measures'' for integration of functions defined on sets \cite{BaGrWr}.
The resulting \emph{Hadwiger integrals}, which generalize both the Lebesgue integral and the Euler integral, provide various notions of the size of a function.
The goal of this paper is to combine Hadwiger integrals with Gaussian-related random fields, obtaining expected-value results that will be useful in various applications.

We briefly describe the potential utility of the Hadwiger integrals.
While the Euler integral is useful for counting the number of objects detected by a sensor network \cite{BG:pnas, BG:siam}, the higher-dimensional Hadwiger integrals return the aggregate perimeter, surface area, etc.\ of the detected objects \cite{BaGrWr}.
Current work (in progress by the author) explores uses of these integrals in image processing.
For example, local Euler/Hadwiger integrals seem to be useful in distinguishing between textures within an image.
While the Euler integral can compare textures in a scale-invariant way, the Hadwiger integrals can help distinguish scale-dependent features in an image.
Furthermore, in cell structures that evolve by a process of mean curvature flow, the intrinsic volumes appear in the three-dimensional von Neumann--Mullins relation, which gives the rate of change of cell volume \cite{MS}.
This suggests that the Hadwiger integrals may be useful for understanding change in functions defined on a cell structure (for example, a function that gives the temperature at each point in a cell).

In applications such as these, one often encounters noise and other uncertainties which can be modeled by random fields.
The data obtained from a sensor network, for example, only approximates a function whose true value cannot be measured exactly at every point in its domain.
The observed function might contain noise or uncertainties that can be modeled by a Gaussian process.
Thus, an understanding of Hadwiger integrals of random fields helps explain the contribution of noise to situations in which one desires to compute Hadwiger integrals.

Bobrowski and Borman gave the expected Euler integral of Gaussian-related random fields \cite{BoBo}.
We generalize their results, computing the expected Hadwiger integrals of Gaussian-related random fields.
We first provide some background material, both on Hadwiger integration and on random fields.
Our main result is Theorem \ref{thm:ExpHadInt}, which gives a formula for the expected Hadwiger integrals in terms of the intrinsic volumes and Gaussian Minkowski functionals.
We then provide two examples in which the Gaussian Minkowski functionals, and thus the expected Hadwiger integrals, are computed explicitly.
Lastly, we connect this work to Hadwiger's theorem, which allows us to obtain the expected values of more general valuations on Gaussian-related random fields.

\section{Hadwiger Integration}

The intrinsic volumes\footnote{The intrinsic volumes are also known as \emph{Lipschitz-Killing curvatures} and, with different normalization, \emph{Minkowski functionals} and \emph{quermassintegrals}. Klain and Rota give a combinatorial approach to the intrinsic volumes \cite{KlRo}, while Adler and Taylor provide a perspective from integral geometry \cite{AdTa2007}.}
are a class of $n+1$ valuations defined on ``tame'' subsets of Euclidean space $\real^n$.
We denote the intrinsic volumes as $\mu_0, \ldots, \mu_n$, with $\mu_0$ the Euler characteristic and $\mu_n$ Lebesgue measure on $\real^n$.
Intuitively, $\mu_k$ gives a notion of the $k$-dimensional size of a set; for example, $\mu_1$ gives the length, or more properly \emph{mean width}, of a set, and $\mu_{n-1}$ is proportional the surface area for $n$-dimensional sets.
The intrinsic volumes appear in the Steiner Formula, which gives the $n$-dimensional volume of a ``tube'' of radius $\rho$ around a closed convex $n$-dimensional set $A$:
\begin{equation}\label{eq:SteinerFormula}
	\mu_n(A + \rho B_n) = \sum_{i=0}^n \omega_{n-i}\rho^{n-i}\mu_i(A).
\end{equation}
In equation \eqref{eq:SteinerFormula}, $B_n$ is the $n$-dimensional unit ball in $\real^n$, $\omega_n$ denotes its volume, and the sum on the left is the Minkowski sum.
The intrinsic volumes are \emph{additive}, meaning that $\mu_i(A) + \mu_i(B) = \mu_i(A \cap B) + \mu_i(A \cup B)$.
Additivity is the key property that facilitates use of the intrinsic volumes as ``measures'' for integration.
Other important properties include invariance with respect to isometries, normalization independent of the dimension of the ambient space, and homogeneity: $\mu_i(\lambda \cdot A) = \lambda^i \mu_i(A)$ for $\lambda > 0$.

To avoid fractals and other exotic sets, we restrict to ``tame'' sets in the sense of \emph{o-minimal geometry} \cite{vdD}.
Briefly, an o-minimal structure is a sequence $(S_n)_n$ of Boolean algebras of subsets of $\real^n$ that satisfy a few axioms: closure under products and projections, $S_n$ contains the diagonal elements $\{ (x_1, \ldots, x_n) \mid x_i=x_j \}$, and $S_1$ consists of finite unions of points and intervals.
Examples of o-minimal structures include the semilinear sets and semialgebraic sets.
Elements of an o-minimal structure are considered \emph{tame}.
In particular, the intrinsic volumes are well-defined for tame sets.

For a function $f: M \to \real$ on a tame set $M \subseteq \real^n$, we write $\{ f \ge s\}$ to denote the superlevel set $\{ p \in M : f(p) \ge s \}$, and similarly for $\{ f > s \}$ and other inequalities.
For the purpose of this paper, a continuous function $f: M \to \real$ is \emph{tame} if the intrinsic volumes of $\{ f \ge s \}$ and $\{ f \le s \}$ are well-defined, except for at most finitely many $s \in \real$.

Integrals with respect to the intrinsic volumes are known as \emph{Hadwiger integrals}, which can be thought of as valuations on functions.
The Hadwiger integrals of continuous functions appear in dual pairs, called the \emph{lower} and \emph{upper} Hadwiger integrals \cite{BaGrWr}.
We denote the lower Hadwiger integral of a function $f : M \to \real$ with respect to $\mu_i$ as $\int f \dmufloor{i}$, and the upper Hadwiger integral as $\int f \dmuceil{i}$.

\begin{definition}\label{HadIntDef}
	The lower and upper Hadwiger integrals of a tame function $f : M \to \real$ are defined as follows:
	\begin{align}
		\int_M f \dmufloor{i} &= \int_0^\infty \left( \mu_i\{f \ge s\} - \mu_i\{f < -s\} \right) ds \\
		\int_M f \dmuceil{i} &= \int_0^\infty \left( \mu_i\{f > s\} - \mu_i\{f \le -s\} \right) ds
	\end{align}
	for any $i \in \{ 0, 1, \ldots, n\}$. 
\end{definition}

The Hadwiger integrals provide various notions of the \emph{size} of a function.
Just as we can interpret $\mu_k$ as indicating the $k$-dimensional size of sets, an integral with respect to $\mu_k$ gives a notion of the $k$-dimensional size of real-valued functions defined on sets.
Intuitively, an integral with respect to $\mu_k$ returns a weighted sum of the $k$-dimensional sizes of all superlevel sets of a function.
For example, if a tame function $f : \real^n \to \zed_{\ge 0}$ has compact support and \emph{finite} image, then
\begin{equation}\label{eq:LowUp}
	\int_{\real^n} f \dmufloor{i} = \int_{\real^n} f \dmuceil{i} = \sum_{s = 1}^{\mathrm{max}(f)} \mu_i\{f \ge s \}.
\end{equation}
While the lower and upper Hadwiger integrals agree (as in \eqref{eq:LowUp}) on functions with finite image, for continuous functions the two integrals are generally not equal.
With suitable assumptions about continuity, any Euclidean-invariant valuation on real-valued functions is a linear combination of the Hadwiger integrals \cite{BaGrWr}.

\section{Random Fields}

A \emph{random field} is a stochastic process, defined over a topological space, taking values in $\real^k$.
Intuitively, a random field $f$ can be thought of as a function on a topological space $M$ whose value at any point $p \in M$ is a random variable.
Adler and Taylor provide a formal definition \cite{AdTa2007, Taylor}.

In particular, we are interested in \emph{Gaussian} random fields.
If the finite-dimensional distributions of $(f(p_1), \ldots, f(p_j))$ are multivariate Gaussian for each $1 \le j < \infty$ and each $(p_1, \ldots, p_j) \in M^j$, then $f$ is a \emph{Gaussian random field}.
Associated to any random field $f : M \to \real^k$ are two important functions: the mean function $m(p) = \expect(f(p))$ and the covariance function
\begin{equation*}
	C(p_1,p_2) = \expect[(f(p_1)-m(p_1))(f(p_2)-m(p_2))].
\end{equation*}
Indeed, the distribution of a real-valued Gaussian random field is completely determined by its mean and covariance functions. 
A random field is \emph{isotropic} if its covariance function is invariant under isometries of $M$.

\begin{example}
	Suppose we model temperature $T$ at position $p$ and in a classroom $M \subset \real^3$.
	Every measurement involves some error, so we can model the temperature as $T(p) = u(p) + f(p)$, 
	where $u$ is the true unknown temperature and $f$ is the measurement error.
	The measurement error at any point can be modeled as a Gaussian random variable, so $f$ can be modeled as a Gaussian random field on $M$.
	(Example inspired by Chung \cite{Chung}.)
\end{example}

We also work with certain non-Gaussian random fields. 
Let $f : M \to \real^k$ be a Gaussian random field, and let $F: \real^k \to \real$ be a function.
Then we call the random field $F \circ f : M \to \real$ a \emph{Gaussian-related} random field.
Of course, the Gaussian case is recovered if $k=1$ and $F(x)=x$.

\section{Gaussian Kinematic Formula}

Adler and Taylor introduced the \emph{Gaussian kinematic formula} (GKF), which gives the expected intrinsic volume of an excursion set of a random field \cite[Theorem 15.9.4]{AdTa2007}.
We state here a version of this formula and give a brief explanation of the quantities involved.

\begin{theorem}[GKF]\label{thm:gfk}
	Let $M$ be a compact regular stratified space.
	Let $f = (f_1, \ldots, f_k) : M \to \real^k$ be a Gaussian random field, with i.i.d.\ components with zero mean and unit variance, and such that with probability one $f_j$ is a stratified Morse function.
	Let $D \subset \real^k$ be closed.
	Then, for $0 \le i \le \mathrm{dim}(M)$,
	\begin{equation}\label{eq:gkf}
		\expect\left( \mu_i\left(f^{-1}(D)\right) \right) = \sum_{j=0}^{\mathrm{dim}(M)-i} \FLAG{i+j}{j} (2\pi)^{-j/2} \mu_{i+j}(M) \gmf{j}(D)
	\end{equation}
	where $\FLAG{i+j}{j}$ is a flag coefficient and $\gmf{j}$ is a Gaussian Minkowski functional.
\end{theorem}

The flag coefficients of Klain and Rota are somewhat analogous to the binomial coefficients, as the notation suggests \cite{KlRo}.
They are defined
\begin{equation}
	\FLAG{n}{m} = \binom{n}{m}\frac{\omega_n}{\omega_m\omega_{n-m}},
\end{equation}
where $\omega_n$ denotes the $n$-dimensional volume of the unit ball in $\real^n$.
As the binomial coefficient $\binom{n}{m}$ counts the number of $m$-element subsets of an $n$-element set, the flag coefficient $\FLAG{n}{m}$ gives a total measure of the $m$-dimensional linear subspaces of $\real^n$.
This measure is important in the definition of the intrinsic volumes to make them \emph{intrinsic} to sets and independent of the ambient space in which a set may be embedded.

The definition of the Gaussian Minkowski functionals $\gmf{j}$ can be found in \cite{AdTa2007, Taylor}.    
They satisfy a tube formula similar to the Steiner Formula \eqref{eq:SteinerFormula}, but involving the Gaussian measure:
\begin{equation}\label{eq:gmf_tube}
	\gamma_n(A + \rho B_n) = \gamma_n(A) + \sum_{j=1}^\infty \frac{\rho^j}{j!} \gmf{j}(A),
\end{equation}
where $\gamma_n$ is the Gaussian measure on $\real^n$.
Since the $\gmf{j}$ do not depend on the dimension of the Gaussian measure space, it is not necessary to write $\mathcal{M}^{\gamma_n}_j$.

In equation \eqref{eq:gkf}, the intrinsic volumes are computed with respect to a Riemannian metric determined by the random field.
This metric is related to the covariance function $C$ of $f$, and is defined
\begin{equation}\label{eq:metric}
	g_p(X_p,Y_p) = \expect[ (X_pf) \cdot (Y_pf) ] = X_pY_qC(p,q) |_{p=q},
\end{equation}
where $X_p, Y_p \in T_pM$, the tangent manifold to $M$ at $p$ \cite[Section 12.2]{AdTa2007}.
In particular, if $f$ is isotropic, then this metric is the Euclidean metric, up to a constant multiple.

Adler and Taylor give a formal definition of regular stratified space \cite[Section 9.2.3]{AdTa2007}; examples include closed manifolds and compact manifolds with boundary.
The assumption that each $f_j$ is a stratified Morse function is also not too restrictive.
Adler and Taylor address this assumption, giving conditions under which it holds \cite[Section 11.3]{AdTa2007}.

\section{Expected Hadwiger Integral of a Random Field}

Bobrowski and Borman computed the expected Euler integral of a Gaussian-related random field \cite{BoBo}. 
We similarly compute the expected Hadwiger integral of such a field.

\begin{theorem}\label{thm:ExpHadInt}
	Let $M$ be an $n$-dimensional compact regular stratified space.
	Let $f: M \to \real^k$ be a Gaussian random field satisfying the GKF conditions.
	Let $F: \real^k \to \real$ be a piecewise $C^2$ function.
	Let $g = F \circ f$, so $g: M \to \real$ is a Gaussian-related random field.
	Then the expected lower Hadwiger integral of $g$ is:
	\begin{equation}\label{eq:ExpHadInt}
		\expect \left(\int_M g \dmufloor{i}\right) = \mu_i(M)\expect(g) + \sum_{j=1}^{n-i} \FLAG{i+j}{j} (2\pi)^{-j/2} \mu_{i+j}(M) \int_\real \gmf{j}\{ F \ge s\} \, ds,
	\end{equation}
	and similarly for the upper Hadwiger integral.
\end{theorem}

\begin{proof}
	From Definition \ref{HadIntDef} of the lower Hadwiger integral,
	\begin{equation*}
		\int_M g \dmufloor{i} = \int_0^\infty \mu_i\{g \ge s\} \,ds - \int_{-\infty}^0 \left( \mu_i(M) - \mu_i\{g \ge s\} \right) ds.
	\end{equation*}
	Taking expectations,
	\begin{equation*}
		\expect \left(\int_M g \dmufloor{i}\right) = \int_0^\infty \expect\left(\mu_i\{g \ge s\}\right) ds - \int_{-\infty}^0 \left( \mu_i(M) - \expect\left(\mu_i\{g \ge s\}\right) \right) ds.
	\end{equation*}
	
	Thus, the expected Hadwiger integral can be expressed in terms of the expected intrinsic volumes of the superlevel sets $\{ g \ge s \}$.
	Since $g =  F \circ f$, it follows that $\{g \ge s\} = f^{-1}\{ F \ge s\}$.
	The Gaussian Kinematic Formula \eqref{eq:gkf} allows us to rewrite the expected intrinsic volumes of $\{ g \ge s \}$:
	\begin{equation*}
		\expect\left(\mu_i\{g \ge s\}\right) = \expect\left( \mu_i\left(f^{-1}\{ F \ge s\}\right)\right) = \sum_{j=0}^{n-i} \FLAG{i+j}{j} (2\pi)^{-j/2} \mu_{i+j}(M) \gmf{j}\{ F \ge s\}.
	\end{equation*}
	
	Combining the previous two equations, we obtain
	\begin{multline*}
		\expect \left(\int_M g \,\dmufloor{i}\right)
		= \int_0^\infty \left( \sum_{j=0}^{n-i} \FLAG{i+j}{j} (2\pi)^{-j/2} \mu_{i+j}(M) \gmf{j}\{ F \ge s\} \right) ds \\
		- \int_{-\infty}^0 \left( \mu_i(M) - \sum_{j=0}^{n-i} \FLAG{i+j}{j} (2\pi)^{-j/2} \mu_{i+j}(M) \gmf{j}\{ F \ge s\} \right) ds.
	\end{multline*}
	
	To make the above expression more manageable, we separate the $j=0$ terms out of the sum, and the expected Hadwiger integral becomes
	\begin{multline*}
		\expect \left(\int_M g \,\dmufloor{i}\right)
		= \mu_i(M)\left[ \int_0^\infty \gmf{0}\{ F \ge s\} \,ds - \int_{-\infty}^0\left( 1 - \gmf{0} \{ F \ge s\} \right) ds \right] \\
		+ \sum_{j=1}^{n-i} \int_{-\infty}^\infty \FLAG{i+j}{j} (2\pi)^{-j/2} \mu_{i+j}(M) \gmf{j}\{ F \ge s\} \,ds.
	\end{multline*}
	
	To simplify the quantity in square brackets above, let $X$ be a $k$-dimensional standard Gaussian random variable, and let $Y = F(X)$.
	Then the definition of the Gaussian Minkowski functionals implies $\gmf{0}\{ F \ge s\} = \prob\left(X \in \{ F \ge s\} \right) = \prob(Y \ge s)$.
	It follows that
	\begin{align*}
		\int_0^\infty \gmf{0}\{ F \ge s\} \,ds - \int_{-\infty}^0 \left( 1 - \gmf{0}\{ F \ge s\} \right) ds 
		&= \int_0^\infty \prob(Y \ge s)\,ds - \int_{-\infty}^0 (1-\prob(Y \ge s))\,ds \\
		&= \int_0^\infty \prob(Y \ge s)\,ds - \int_{-\infty}^0 \prob(Y < s)\,ds \\
		&= \expect(Y) = \expect(g),
	\end{align*}
	where the last equality holds because $f$ is standard normal at every point.
	Thus, we obtain equation \eqref{eq:ExpHadInt}.
\end{proof}

\section{Examples}

We give two examples to illustrate the computation of $\gmf{j}\{ F \ge s \}$ in equation \eqref{eq:ExpHadInt}.
These examples extend those given by Bobrowski and Borman for the Euler case \cite[Section 4]{BoBo} and involve computations by Adler and Taylor \cite[Section 15.10]{AdTa2007}.

\begin{example}[The Real Case]
	Suppose random field $f: M \to \real$ satisfies the conditions of Theorem \ref{thm:ExpHadInt}, $F: \real \to \real$ is piecewise $C^2$, and $g = F \circ f$.
	In this case it is possible to simplify the $\gmf{j}\{ F \ge s \}$ from equation \eqref{eq:ExpHadInt}.
	
	By continuity of $F$, $\{ F \ge s \}$ can be written as a disjoint union of closed intervals:
	\begin{equation*}
		\{ F \ge s \} = F^{-1}[s, \infty) = \bigcup_i [a_i, b_i],
	\end{equation*}
	where one of the $a_i$ may be $-\infty$ and one of the $b_i$ may be $\infty$.
	Let $\varphi(x) = (2\pi)^{-1/2} e^{-x^2/2}$ be the standard Gaussian density and let $H_m(x) = (-1)^m \varphi(x)^{-1} \frac{d^m}{dx^m} \varphi(x)$ be the $m^\textrm{th}$ Hermite polynomial.
	Bobrowski and Borman show \cite[Section 4.1]{BoBo} that for $j \ge 1$,
	\begin{equation*}
		\gmf{j}\{ F \ge s \} = \sum_i \left( (-1)^{j-1} H_{j-1}(b_i)\varphi(b_i) + H_{j-1}(a_i)\varphi(a_i) \right),
	\end{equation*}
	and furthermore that any infinite $a_i$ or $b_i$ affect only $\gmf{0}$.
	Thus, we assume all $a_i$ and $b_i$ are finite and $F^{-1}(s) = \bigcup_i \{ a_i, b_i \}$.
	Since $F'(a_i) > 0$ and $F'(b_i) < 0$, we obtain
	\begin{equation*}
		\gmf{j}\{ F \ge s\} = \sum_{x \in F^{-1}(s)} \left( \mathrm{sign}\left( F'(x) \right)\right)^{j-1} H_{j-1}(x) \varphi(x).
	\end{equation*}
	We can then express the expected lower Hadwiger integral of $g$ as:
	\begin{multline}
		\expect \left(\int_M g \dmufloor{i}\right) = \mu_i(M)\expect(g) + \\
		\sum_{j=1}^{n-i} \FLAG{i+j}{j} (2\pi)^{-j/2} \mu_{i+j}(M) \int_\real \sum_{x \in F^{-1}(s)} \left( \mathrm{sign}\left( F'(x) \right)\right)^{j-1} H_{j-1}(x) \varphi(x) \, ds.
	\end{multline}
	
	As a further special case, if $F(x)=x$, then $g$ is a Gaussian random field, $\expect(g)=0$, and we have:
	\begin{equation}
		\expect \left(\int_M g \dmufloor{i}\right) = \sum_{j=1}^{n-i} \FLAG{i+j}{j} (2\pi)^{-j/2} \mu_{i+j}(M) \int_\real H_{j-1}(s) \varphi(s) \, ds.
	\end{equation}
\end{example}

\begin{example}[The $\chi^2$ Case]
	Let $M$ be a compact $n$-dimensional manifold, $f: M \to \real^k$ a random field satisfying the conditions of Theorem \ref{thm:ExpHadInt} with $k \ge n$, and $F(x_1, \ldots, x_k) = \sum_{i=1}^k x_i^2$.
	Then $g = F \circ f$ is a called a $\chi^2$ random field.
	
	If $s \le 0$, then $\{F \ge s\} = \real^k$. 
	The tube formula \eqref{eq:gmf_tube} then implies that for $s \le 0$ and $j \ge 1$, $\gmf{j}\{F \ge s\} = 0$.
	Thus, it suffices to consider positive $s$.
	
	In the $\chi^2$ case, Adler and Taylor show \cite[Section 15.10.2]{AdTa2007} that for $j \ge 1$,
	\begin{equation*}
		\gmf{j}\{F \ge s\} = (-1)^{j-1} \left. \frac{d^{j-1}p_k(x)}{dx^{j-1}} \right|_{x=\sqrt{s}} \quad \text{where} \quad p_k(x) = \frac{x^{k-1}e^{-x^2/2}}{\Gamma\left(\frac{k}{2}\right) 2^{(k-2)/2}}.
	\end{equation*}
	Integrating, we obtain:
	\begin{align*}
		\int_0^\infty \gmf{1}\{F \ge s\} \, ds &= 2\sqrt{2}\:\frac{\Gamma\left(\frac{k+1}{2}\right)}{\Gamma\left(\frac{k}{2}\right)}, \\
		\int_0^\infty \gmf{2}\{F \ge s\} \, ds &= 2, \\
		\int_0^\infty \gmf{j}\{F \ge s\} \, ds &= 0 \quad \text{for } 3 \le j \le n.
	\end{align*}
	Additionally, $\expect(g) = k$. 
	Therefore, we can express the expected lower Hadwiger integral of $g$ as:
	\begin{equation}
		\expect \left(\int_M g \dmufloor{i}\right) = k\mu_i(M) + \FLAG{i+1}{1} \frac{2 \mu_{i+1}(M)}{\sqrt{\pi}} \cdot \frac{\Gamma\left(\frac{k+1}{2}\right)}{\Gamma\left(\frac{k}{2}\right)} + \FLAG{i+2}{2}\frac{\mu_{i+2}(M)}{\pi}.
	\end{equation}

\end{example}

\section{Connection to Hadwiger's Theorem}

We now combine Theorem \ref{thm:ExpHadInt} with Hadwiger's Theorem to obtain expected values of more general valuations of Gaussian-related random fields.
The classic Hadwiger Theorem states that all Euclidean-invariant convex-continuous valuations on subsets of $\real^n$ are linear combinations of the intrinsic volumes \cite{KlRo}.
Recent work lifted the theorem from valuations on sets to valuations on functions defined on sets, obtaining Hadwiger's Theorem for Functions \cite[Theorem 14]{BaGrWr}.

Hadwiger's Theorem for Functions requires the dual notions of lower- and upper-continuous valuations on functions.
For a rigorous treatment in integral-geometric terms, see \cite[Definition 8]{BaGrWr}.
Briefly, a valuation $v$ on functions assigns a real number to each function such that $v(0)=0$ and the following additivity condition is satisfied:
\begin{equation*}
	v(f) + v(g) = v(f \vee g) + v(f \wedge g),
\end{equation*}
for tame functions $f$ and $g$, where $\vee$ and $\wedge$ denote pointwise max and min, respectively.
If valuation $v$ is lower-continuous, then $\lim_{m \to \infty} v\left( \frac{1}{m} \lfloor mf \rfloor \right) = v(f)$, where $\lfloor \cdot \rfloor$ is the floor function, and dually for upper-continuity.
The duality of lower- and upper-continuity mirrors that present in the lower and upper Hadwiger integrals.
Hadwiger's Theorem for Functions is then \cite[Theorem 14]{BaGrWr}:

\begin{theorem}[Hadwiger's Theorem for Functions]\label{thm:HadFun}
	If $v$ is a Euclidean-invariant, lower-continuous valuation on tame functions $f: \real^n \to \real$, then
	\begin{equation}
		v(f) = \sum_{i=0}^n \int_{\real^n} c_i \circ f \dmufloor{i}
	\end{equation}
	for some continuous and monotone functions $c_i \in C(\real)$ satisfying $c_i(0)=0$.
	Similarly, an upper-continuous valuation can be written in terms of upper Hadwiger integrals. 
\end{theorem}

We introduce the concept of a piecewise $C^2$ valuation, which imposes a smoothness condition on the functions $c_i$ in the previous theorem.

\begin{definition}
	A Euclidean-invariant, lower- or upper-continuous valuation $v$ is a \emph{piecewise $C^2$ valuation} if the functions $c_i$ guaranteed by Theorem \ref{thm:HadFun} are piecewise $C^2$ functions.
\end{definition}

Determining whether a valuation $v$ is piecewise $C^2$ is straightforward.
Let $A_0, A_1, \ldots, A_n$ be a sequence of subsets of $\real^n$ such that $\mu_i(A_j) = \delta_{ij}$, where $\delta_{ij}$ is the Kronecker delta.
Then the indicator function $h_r = r\One_{A_j}$ is a test function that isolates $c_j(r)$:
\begin{equation}
	v(h_r) = v(r\One_{A_j}) = \sum_{i=0}^n \int_{\real^n} c_i(r\One_{A_j}) \dmufloor{i} = \sum_{i=0}^n c_i(r) \mu_i(A_j) = c_j(r).
\end{equation}

For a piecewise $C^2$ valuation of a Gaussian-related random field, Theorem \ref{thm:ExpHadInt} gives the expected values of the Hadwiger integrals that appear in Hadwiger's theorem.
We obtain the following corollary.

\begin{corollary}
	Let $v$ be a lower-continuous piecewise $C^2$ valuation, and let $g = F \circ f : M \to \real$ be a Gaussian-related random field as in Theorem \ref{thm:ExpHadInt}, with the additional requirement that $f$ is isotropic.
	Then the expected value of $v(g)$ is
	\begin{equation}\label{eq:ExpVal}
		\expect(v(g)) = \sum_{i=0}^{n} \left( \mu_i(M)\expect(c_i(g)) + \sum_{j=1}^{n-i} \FLAG{i+j}{j} (2\pi)^{-j/2} \mu_{i+j}(M) \int_\real \gmf{j}\{ c_i(F) \ge s\} \, ds \right),
	\end{equation}
	and similarly for an upper-continuous piecewise $C^2$ valuation.
\end{corollary}
\begin{proof}
	Since the field is isotropic, the $\mu_i$ are calculated with respect to the Euclidean metric.
	Thus, Euclidean-invariance allows us to apply Theorem \ref{thm:HadFun}, obtaining the decomposition
	\begin{equation}\label{eq:vg}
		v(g) = \sum_{i=0}^n \int_{\real^n} c_i \circ g \dmufloor{i}.
	\end{equation}
	Since each of the $c_i$ are piecewise $C^2$ functions, each composition $c_i(g)$ is a Gaussian-related random field satisfying the conditions of Theorem \ref{thm:ExpHadInt}.
	Therefore, we can apply Theorem \ref{thm:ExpHadInt} to each summand in equation \eqref{eq:vg}, obtaining equation \eqref{eq:ExpVal}.
\end{proof}

We conclude with a comment about critical values.
While the Euler integral has an elegant expression in terms of the critical values of a random field \cite{BG:pnas, BoBo}, 
a similar phenomena for the more general Hadwiger integrals is elusive.
Because Euler characteristic is a topological invariant, and the topology of superlevel sets of a function changes only at critical values, the Euler integral is determined precisely by the critical values.
However, the other intrinsic volumes are metric-dependent, returning geometric information about sets.
Thus, it appears that the Hadwiger integrals, other than the Euler integral, cannot be reduced to critical values alone.

\vspace{20pt}

\begin{center}
	\footnotesize{\textsc{Acknowledgements}}
\end{center}
\vspace{-3pt}
\footnotesize{
The work presented in this paper was partially carried out while the author was a postdoctoral fellow at the Institute for Mathematics and its Applications, 
during the annual program on Scientific and Engineering Applications of Algebraic Topology.
In particular, the author thanks Robert Adler and Jonathan Taylor for helpful conversations during the tutorial \em{Introduction to Statistics and Probability for Topologists}.
}


\bibliographystyle{amsalpha}

\end{document}